\documentclass[12pt,oneside,english]{amsart}
\usepackage[T1]{fontenc}
\usepackage[latin9]{inputenc}
\usepackage{geometry}
\usepackage{mathtools}
\usepackage{amstext}
\usepackage{amsthm}
\usepackage{amssymb}

\makeatletter

\providecommand{\tabularnewline}{\\}

\numberwithin{equation}{section}
\numberwithin{figure}{section}
\theoremstyle{plain}
\newtheorem{thm}{\protect\theoremname}
\theoremstyle{definition}
\newtheorem{defn}{\protect\definitionname}
\theoremstyle{plain}
\newtheorem{lem}{\protect\lemmaname}
\theoremstyle{remark}
\newtheorem{rem}{\protect\remarkname}
\theoremstyle{plain}
\newtheorem{prop}{\protect\propositionname}
\theoremstyle{definition}
\newtheorem{problem}{\protect\problemname}

\DeclareMathOperator{\slicerank}{slice-rank}
\DeclareMathOperator{\sgn}{sgn}
\DeclareMathOperator{\prank}{partition-rank}

\makeatother

\usepackage{babel}
\providecommand{\definitionname}{Definition}
\providecommand{\lemmaname}{Lemma}
\providecommand{\problemname}{Problem}
\providecommand{\propositionname}{Proposition}
\providecommand{\remarkname}{Remark}
\providecommand{\theoremname}{Theorem}

\begin{document}
\title{The Chromatic Number of $\mathbb{R}^{n}$ with Multiple Forbidden
Distances}
\author{Eric Naslund}
\begin{abstract}
Let $A\subset\mathbb{R}_{>0}$ be a finite set of distances, and let
$G_{A}(\mathbb{R}^{n})$ be the graph with vertex set $\mathbb{R}^{n}$
and edge set $\{(x,y)\in\mathbb{R}^{n}:\ \|x-y\|_{2}\in A\}$, and
let $\chi(\mathbb{R}^{n},A)=\chi\left(G_{A}(\mathbb{R}^{n})\right)$.
Erd\H{o}s asked about the growth rate of the $m$-distance chromatic
number
\[
\overline{\chi}(\mathbb{R}^{n};m)=\max_{|A|=m}\chi(\mathbb{R}^{n},A).
\]
We improve the best existing lower bound for $\overline{\chi}(\mathbb{R}^{n};m)$,
and show that 
\[
\overline{\chi}(\mathbb{R}^{n};m)\geq\left(\Gamma_{\chi}\sqrt{m+1}+o(1)\right)^{n}
\]
where $\Gamma_{\chi}=0.79983\dots$ is an explicit constant. Our full
result is more general, and applies to cliques in this graph. Let
$\chi_{k}(G)$ denote the minimum number of colors needed to color
$G$ so that no color contains a $(k+1)$-clique, and let $\overline{\chi}_{k}(\mathbb{R}^{n};m)$
denote the largest value this takes for any distance set of size $m$
. Using the Partition Rank Method, we show that 
\[
\overline{\chi}_{k}(\mathbb{R}^{n};m)>\left(\Gamma_{\chi}\sqrt{\frac{m+1}{k}}+o(1)\right)^{n}.
\]
\end{abstract}

\email{naslund.math@gmail.com}
\maketitle

\section{Introduction}

The chromatic number of Euclidean space, $\chi(\mathbb{R}^{n})$,
is the minimum number of colors required to color $\mathbb{R}^{n}$
so that no two points at distance $1$ have the same color. When $n=2$
determining $\chi(\mathbb{R}^{2})$ is known as the Hadwiger-Nelson
problem \cite{Gardner1960BookofProblems,Hadwiger1944,Hadwiger1945},
and the best existing bounds are 
\[
5\leq\chi\left(\mathbb{R}^{2}\right)\leq7,
\]
where the lower bound is a recent improvement due to De Grey \cite{DeGrey2018ChromaticNumberOfThePlane}.
For large $n$, Frankl and Wilson proved that $\chi(\mathbb{R}^{n})$
grows exponentially \cite{FranklWilson1981ChromaticNumberOfASpace},
and the best bounds are
\[
(1.239\dots+o(1))^{n}<\chi\left(\mathbb{R}^{n}\right)\leq\left(3+o(1)\right)^{n},
\]
due to Raigorodskii \cite{Raigorodskii2000OnTheChromaticNumberOfASpace},
and Larman and Rogers \cite{LarmanRogers1972DistancesWithinEuclideanSpaces},
respectively.

Erd\H{o}s \cite{Erdos1985ProblemsInCombinatorialGeometryChromaticMultiDistancePosed}
proposed a variant where the graph is based on multiple distances
instead of one. For a set of distances $A$, let $G_{A}(\mathbb{R}^{n})$
denote the graph with vertex set $\mathbb{R}^{n}$ and edge set $\{(x,y)\in\mathbb{R}^{n}:\ \|x-y\|_{2}\in A\}$,
and let 
\[
\chi(\mathbb{R}^{n},A)=\chi\left(G_{A}(\mathbb{R}^{n})\right).
\]
The $m$-distance chromatic number of $\mathbb{R}^{n}$ is 
\[
\overline{\chi}\left(\mathbb{R}^{n};m\right)=\max_{A:\ |A|=m}\chi\left(\mathbb{R}^{n},A\right).
\]
Note that by rescaling, for any set $A$ with $|A|=1$, we have that
\[
\overline{\chi}\left(\mathbb{R}^{n};1\right)=\chi\left(\mathbb{R}^{n};A\right)=\chi\left(\mathbb{R}^{n}\right).
\]
Erd\H{o}s noted that $\frac{1}{m}\overline{\chi}(\mathbb{R}^{2};m)\rightarrow\infty$
as $m\rightarrow\infty$ \cite{Soifer2011BookChapterChromaticNumberAndItsHistoryChapter},
and by considering the $n\times n$ grid in the plane, which has $Cn^{2}/\sqrt{\log n}$
distinct distances, it follows that for some constant $C>0$, 
\begin{equation}
\overline{\chi}(\mathbb{R}^{2};m)\geq Cm\sqrt{\log m}.\label{eq:Erdos_R2_lower_bound}
\end{equation}
Erd\H{o}s asked whether $\overline{\chi}\left(\mathbb{R}^{2};m\right)$
grows polynomially, and more generally whether we can determine $\overline{\chi}(\mathbb{R}^{n};m)$
for any $m$, and $\chi(\mathbb{R}^{n},A)$ for a given $A$ \cite[Open Problem 10 and 42]{Soifer2011BookChapterChromaticNumberAndItsHistoryChapter}. 

In higher dimensions, the Larman-Rogers bound implies the upper bound
\[
\overline{\chi}\left(\mathbb{R}^{n};m\right)\leq\left(3+o_{n}(1)\right)^{mn},
\]
(see \cite[page 740]{Kupavskii2010ChromaticNumberForbiddenDistances})
and Raigorodskii \cite{Raigorodskii2001BorsuksProblemGeneralmLowerBoundReference}
proved that there exists $c_{1},c_{2}>0$ such that 
\[
\overline{\chi}\left(\mathbb{R}^{n};m\right)>(c_{1}m)^{c_{2}n}.
\]
Berdnikov \cite{Berdnikov2016EstimateForTheChromaticNumberOfEuclideanSpaceWithSeveral,Berdnikov2017EstimateForTheChromaticNumberOfLQSpace,Berdnikov2018ChromaticNumbersBestBoundsCliques}
improved on Raigorodskii's lower bound, and showed that any $c_{2}<\frac{1}{2}$
is admissible. In this paper we prove the following result: 
\begin{thm}
\label{thm:main_theorem_chi_m}We have that 
\begin{equation}
\overline{\chi}\left(\mathbb{R}^{n};m\right)\geq\left(\Gamma_{\chi}\sqrt{m+1}+o_{n}(1)\right)^{n},\label{eq:main_theorem_lower_bound_chi_m}
\end{equation}
where 
\begin{equation}
\Gamma_{\chi}=\sqrt{\frac{\pi}{2}}\max_{x>0}\frac{1-e^{-x}}{\sqrt{x}}=0.7998308498\dots.\label{eq:main_constant_definition}
\end{equation}
\end{thm}
We prove this theorem by examining the specific set of distances,
$A_{m}=\{1,\sqrt{2},\sqrt{3},\dots,\sqrt{m}\}$, and showing that
\begin{equation}
\chi\left(\mathbb{R}^{n},A_{m}\right)\geq\left(\Gamma_{\chi}\sqrt{m+1}+o(1)\right)^{n}.\label{eq:chi_rn_specific_distance_set_lower_bound}
\end{equation}
All of the best lower bounds for $\overline{\chi}(\mathbb{R}^{n};m)$,
including Erd\H{o}s's bound (\ref{eq:Erdos_R2_lower_bound}), use
this specific set of distances, or a large subset of it. In \cite{Kupavskii2010ChromaticNumberForbiddenDistances}
Kupavskii showed that
\begin{equation}
\chi\left(\mathbb{R}^{n},A_{m}\right)\leq\left(2(\sqrt{m}+1)+o_{n}(1)\right)^{n},\label{eq:kupavskii_upper_bound}
\end{equation}
and so Theorem \ref{thm:main_theorem_chi_m} implies that for the
set of distances $A_{m}$, we understand exactly how $\chi(\mathbb{R}^{n},A_{m})$
grows with $m$.

Our full result is more general, and applies to $m$-distance $k$-cliques.

\subsection{$k$-Cliques}

Let $\chi_{k}(G)$ denote the minimum number of colors needed to color
$G$ so that no color contains a $(k+1)$-clique. For a set of distances
$A$, define
\[
\chi_{k}\left(\mathbb{R}^{n},A\right)=\chi_{k}\left(G_{A}(\mathbb{R}^{n})\right),
\]
where as before, $G_{A}(\mathbb{R}^{n})$ is the graph with vertex
set $\mathbb{R}^{n}$ where two vertices are connected if their distance
is in $A$. Define 
\[
\overline{\chi}_{k}\left(\mathbb{R}^{n};m\right)=\max_{A:\ |A|=m}\overline{\chi}_{k}\left(\mathbb{R}^{n},A\right).
\]
In 1987 Frankl and R\"{o}dl \cite[Theorem 1.18]{FranklRodl1987FranklRodlTheorem}
proved that $\chi_{k}(\mathbb{R}^{n}):=\overline{\chi}_{k}(\mathbb{R}^{n};1)$
grows exponentially with $n$ for any $k$. For $m=1$ and general
$k$ the best lower and upper bounds, due Sagdeev \cite{Sagdeev2018ChromaticNumberSimplexLowerBound}
and Prosanov \cite{Prosanov2018UpperBoundsForTheChromaticNumberCliquesRn},
respectively, are 
\begin{equation}
\left(1+\frac{1}{2^{2^{k+4}}}+o(1)\right)^{n}<\overline{\chi}_{k}\left(\mathbb{R}^{n};1\right)\leq\left(1+\sqrt{\frac{2(k+1)}{k}}+o(1)\right)^{n}.\label{eq:sagdeev_clique_chromatic_lower_bound}
\end{equation}
For $k=2$, the case of an equilateral triangle, better quantitative
bounds are known \cite{Naslund2020MonochromaticEquilateralTriangles}
using the slice rank method developed in \cite{TaosBlogCapsets,CrootLevPachZ4,EllenbergGijswijtCapsets}.
Berdnikov \cite{Berdnikov2018ChromaticNumbersBestBoundsCliques} gave
lower bounds for $\overline{\chi}_{k}$ for $m\rightarrow\infty$
, and proved that for any $c_{2}<\frac{1}{2}$, there exists $c_{1}>0$,
such that for sufficiently large $m$, $\overline{\chi}_{k}(\mathbb{R}^{n};m)>(c_{1}m)^{c_{2}n}$.
Using the Partition Rank Method, introduced in \cite{Naslund2020PartitionRank,Naslund2020EGZPartitionRank}
we provide a quantitative lower bound for the case of cliques with
multiple distances that improves on the best existing bounds.
\begin{thm}
\label{thm:main_theorem_chi_m_k_clique}For $k\geq1$, we have that
\begin{equation}
\chi_{k}\left(\mathbb{R}^{n},A_{m}\right)\geq\left(\Gamma_{\chi}\sqrt{\frac{m+1}{k}}+o(1)\right)^{n},\label{eq:clique_multicolor_lower_bound}
\end{equation}
where $\Gamma_{\chi}=\sqrt{\frac{\pi}{2}}\max_{x>0}\frac{1-e^{-x}}{\sqrt{x}}$.
\end{thm}
For large $m$, the dependence on $k$ in this lower bound is substantially
better than in (\ref{eq:sagdeev_clique_chromatic_lower_bound}). This
is because (\ref{eq:sagdeev_clique_chromatic_lower_bound}) is based
on Frankl and R\"{o}dl's approach \cite[Theorem 1.18]{FranklRodl1987FranklRodlTheorem},
which inductively applies a result for the $k=1$ case. The right
hand side of (\ref{eq:clique_multicolor_lower_bound}) is nontrivial
only when $m+1>\Gamma_{\chi}^{-2}k$, however the method used yields
a nontrivial result for $m\geq k$ (as stated in Theorem \ref{thm:main_theorem_chi_m_k_clique}
above). For $k$ larger than $m$, the best lower bound comes from
the one distance case due to Sagdeev \cite{Sagdeev2018ChromaticNumberSimplexLowerBound}.

\subsection{Outline}

Sections \ref{sec:The-Partition-Rank}, \ref{sec:Lower-Bounds-for-chi_k}
and \ref{sec:Asymptotics} are devoted to the proof of Theorem \ref{thm:main_theorem_chi_m_k_clique}.
In section \ref{sec:The-Partition-Rank}, we provide an introduction
the Partition Rank method of \cite{Naslund2020PartitionRank,Naslund2020EGZPartitionRank},
and define the Distinctness Indicator Function which will be needed
to generalize to $k\geq2$. In section \ref{sec:Lower-Bounds-for-chi_k}
we prove a lower a bound for $\chi(\mathbb{R}^{n},A_{m})$ in terms
of a maximization involving a truncated theta function. Define 
\begin{equation}
\theta(t)=\sum_{l=1}^{\infty}t^{\binom{l}{2}}=1+t+t^{3}+t^{6}+\cdots\label{eq:theta_function_definition}
\end{equation}
and define the truncation
\[
\theta\left(t;l\right)=1+t+t^{3}+t^{6}+t^{10}+\cdots+t^{\binom{l}{2}}.
\]
Using the partition rank method, we prove that
\begin{thm}
\label{thm:main_theorem_specific_max_clique}Let $m\geq1$ be given.
For any $l>1$, and any $k\geq1$, we have that 
\begin{equation}
\chi_{k}\left(\mathbb{R}^{n},A_{m}\right)\geq\left(\max_{0<t<1}\frac{\theta\left(t^{\frac{k}{m+1}};l\right)}{1+t+\cdots+t^{l-1}}+o(1)\right)^{n}.\label{eq:multi_colored_a_m_clique_bound}
\end{equation}
\end{thm}
The right hand size above is non-trivial for any $k\leq m$. In section
\ref{sec:Asymptotics} we analyze the growth rate of the right hand
side of (\ref{thm:max_l_theta_gamma_lower_bound_gamma_chi}), and
prove the following theorem:
\begin{thm}
\label{thm:max_l_theta_gamma_lower_bound_gamma_chi}For $0<\gamma<1$,
\begin{equation}
\max_{l\geq1}\max_{0<t<1}\frac{\theta\left(t^{\gamma};l\right)}{1+t+\cdots+t^{l-1}}\geq\Gamma_{\chi}\sqrt{\frac{1}{\gamma}}\label{eq:theta(t_gamma)_lower_bound}
\end{equation}
where $\Gamma_{\chi}=\sqrt{\frac{\pi}{2}}\max_{0<u<\infty}\frac{1-e^{-u}}{\sqrt{u}}$.
\end{thm}
To prove Theorem \ref{thm:max_l_theta_gamma_lower_bound_gamma_chi},
we first analyze how 
\[
\max_{0<t<1}\frac{\theta\left(t^{\gamma};l\right)}{1+t+\cdots+t^{l-1}}
\]
depends on $l$, and prove that this quantity is bounded from below
by $\max_{0<t<1}(1-t)\theta(t^{\gamma})$. The Poisson summation formula
yields a functional equation for $\theta(t)$, and we use this functional
equation to lower bound $\max_{0<t<1}(1-t)\theta(t^{\gamma})$ and
prove Theorem \ref{thm:max_l_theta_gamma_lower_bound_gamma_chi}.
Theorem \ref{thm:main_theorem_chi_m_k_clique} follows directly from
Theorem \ref{thm:main_theorem_specific_max_clique} and Theorem \ref{thm:max_l_theta_gamma_lower_bound_gamma_chi}.

In section \ref{sec:Open-Problems}, we mention a handful of related
open problems and discuss their significance. 

In the Appendix we use the methods of this paper to prove Theorem
\ref{thm:theta_function_double_cap_conjecture_claim}, stated in the
Open Problems section, which relates a minimization involving the
Theta function of an even integral lattice to the constant in the
Double Cap Conjecture.

\subsection{Exact Maximization}

In section \ref{sec:Asymptotics}, we show that when $k=1$, the right
hand side of equation (\ref{eq:multi_colored_a_m_clique_bound}) is
maximized for some $l\leq2m+1$ (note that for larger $m$, the value
of $l$ that maximizes equation (\ref{eq:multi_colored_a_m_clique_bound})
satisfies $l<2m$). This disproves Conjecture 1 of Gorskaya, Mitricheva,
Protasov, and Raigorodskii \cite{GorskayaMitrichevaProtasovRaigorodskii2009ConvexMinimizationMethods}.
The explicit nature of Theorem \ref{thm:max_l_theta_gamma_lower_bound_gamma_chi}
has some advantages. When $k=m=1$, this quantity is maximized when
$l=3$, and we retrieve Raigorodskii's lower bound
\[
(\chi(\mathbb{R}^{n}))^{\frac{1}{n}}\gtrsim\max_{0<t<1}\frac{1+t+t^{2}}{1+\sqrt{t}+t^{\frac{3}{2}}}=1.23956674\dots,
\]
expressed as a maximization of an explicit rational function. For
$k=1$ and small values of $m$, we improve upon the calculations
of Gorskaya, Mitricheva, Protasov, and Raigorodskii \cite{GorskayaMitrichevaProtasovRaigorodskii2009ConvexMinimizationMethods}
for $\overline{\chi}(\mathbb{R}^{m},m)$. Define 
\[
\zeta_{m}=\limsup_{n\rightarrow\infty}\left(\overline{\chi}(\mathbb{R}^{n};m)\right)^{\frac{1}{n}}.
\]
For $m=2$ and $m=3$ for example, Theorem \ref{thm:main_theorem_specific_max_clique}
yields the lower bounds 
\begin{center}
\begin{tabular}{|c|c|c|}
\hline 
 & Lower bound from Theorem \ref{thm:main_theorem_specific_max_clique} & Lower bound from \cite{GorskayaMitrichevaProtasovRaigorodskii2009ConvexMinimizationMethods}\tabularnewline
\hline 
$m=2$ & $\zeta_{2}\geq1.466299$ & $\zeta_{2}\geq1.465869$\tabularnewline
\hline 
$m=3$ & $\zeta_{3}\geq1.667508$ & $\zeta_{3}\geq1.667462$\tabularnewline
\hline 
\end{tabular}
\par\end{center}

For the case of $k$-simplicies define
\[
\zeta_{m}^{k}=\limsup_{n\rightarrow\infty}\left(\overline{\chi}_{k}(\mathbb{R}^{n};m)\right)^{\frac{1}{n}}.
\]
Theorem \ref{thm:main_theorem_specific_max_clique} yields the following
table for small values of $k,m$:
\begin{center}
\begin{tabular}{|c|c|c|c|c|}
\hline 
 & $k=1$ & $k=2$ & $k=3$ & $k=4$\tabularnewline
\hline 
$m=1$ & $\zeta_{1}^{1}\geq1.239566$ &  &  & \tabularnewline
\hline 
$m=2$ & $\zeta_{2}^{1}\geq1.466299$ & $\zeta_{2}^{2}\geq1.118433$ &  & \tabularnewline
\hline 
$m=3$ & $\zeta_{3}^{1}\geq1.667508$ & $\zeta_{3}^{3}\geq1.239566$ & $\zeta_{3}^{3}\geq1.083024$ & \tabularnewline
\hline 
$m=4$ & $\zeta_{4}^{1}\geq1.848150$ & $\zeta_{4}^{2}\geq1.356230$ & $\zeta_{4}^{3}\geq1.158048$ & $\zeta_{4}^{4}\geq1.063933$\tabularnewline
\hline 
$m=5$ & $\zeta_{5}^{1}\geq2.013079$ & $\zeta_{5}^{2}\geq1.466299$ & $\zeta_{5}^{3}\geq1.239566$ & $\zeta_{5}^{4}\geq1.118433$\tabularnewline
\hline 
\end{tabular}
\par\end{center}

\section{The Partition Rank\label{sec:The-Partition-Rank}}

The Partition Rank method is a multivariable version of the linear
algebraic method, and allows us to prove our main result for $(k+1)$-cliques.
The Partition Rank was introduced by the author in \cite{Naslund2020EGZPartitionRank,Naslund2020PartitionRank},
and we refer the reader to those papers for a more comprehensive overview.
To motivate the definition, we first revisit the definition of \emph{tensor
rank. }For finite sets $X_{1},\dots,X_{n}$, an order $n$ tensor
refers to a function $F:X_{1}\times\cdots\times X_{n}\rightarrow\mathbb{F}$.
This can be thought of as a $|X_{1}|\times\cdots\times|X_{n}|$ grid
of elements of $\mathbb{F}$, and when $n=2$, we call this a matrix,
and when $n=1$, a vector. The function $h\colon X_{1}\times\cdots\times X_{n}\rightarrow\mathbb{F}$
is called a \emph{rank 1 function} if it takes the form
\[
h(x_{1},\dots,x_{n})=f_{1}(x_{1})f_{2}(x_{2})\cdots f_{n}(x_{n}),
\]
and the tensor rank of $F:X_{1}\times\cdots\times X_{n}\rightarrow\mathbb{F}$
is defined to be the minimal $r$ such that 
\[
F=\sum_{i=1}^{r}g_{i}
\]
where the $g_{i}$ are rank 1 functions. Given variables $x_{1},\dots,x_{k}$,
and a set $S\subset\{1,\dots,k\}$, $S=\left\{ s_{1},\dots,s_{m}\right\} $,
let $\vec{x}_{S}$ denote the $m$-tuple 
\[
x_{s_{1}},\dots,x_{s_{m}},
\]
so that for a function $g$ of $m$ variables, we have 
\[
g(\vec{x}_{S})=g(x_{s_{1}},\dots,x_{s_{m}}).
\]
For example, if $k=5$, and $S=\{1,2,4\}$, then $g\left(\vec{x}_{S}\right)=g(x_{1},x_{2},x_{4})$,
and $f\left(\vec{x}_{\{1,\dots,k\}\backslash S}\right)=f(x_{3},x_{5})$.
A partition of $\{1,2,\dots,k\}$ is a collection $P$ of non-empty
pairwise disjoint subsets of $\{1,\dots,k\}$ such that 
\[
\bigcup_{A\in P}A=\{1,\dots,k\}.
\]
We say that $P$ is the \emph{trivial} partition if it consists only
of a single set, $\{1,\dots,k\}$.
\begin{defn}
Let $X_{1},\dots,X_{k}$ be finite sets, and let 
\[
h:X_{1}\times\cdots\times X_{k}\rightarrow\mathbb{F}.
\]
We say that $h$ has \emph{partition rank} $1$ if there exists some
non-trivial partition $P$ of the variables $\{1,\dots,k\}$ such
that 
\[
h(x_{1},\dots,x_{k})=\prod_{A\in P}f_{A}\left(\vec{x}_{A}\right)
\]
for some functions $f_{A}$. We say that $h$ has \emph{slice rank}
$1$, if in addition one of the sets $A\in P$ is a singleton.
\end{defn}
The tensor $h:X_{1}\times\cdots\times X_{k}\rightarrow\mathbb{F}$
will have partition rank $1$ if and only if it can be written in
the form 
\[
h(x_{1},\dots,x_{k})=f(\vec{x}_{S})g(\vec{x}_{T})
\]
for some $f,g$ and some $S,T\neq\emptyset$ with $S\cup T=\{1,\dots,k\}$.
Additionally, $h$ will have slice rank $1$ if it can be written
in the above form with either $|S|=1$ or $|T|=1$. In other words,
a function $h$ has partition rank $1$ if the tensor can be written
as a non-trivial outer product, and it has slice rank $1$ if it can
be written as the outer product between a vector and a $k-1$ dimensional
tensor. 
\begin{defn}
Let $X_{1},\dots,X_{k}$ be finite sets. The \emph{partition rank}
of
\[
F:X_{1}\times\cdots\times X_{k}\rightarrow\mathbb{F},
\]
is defined to be the minimal $r$ such that 
\[
F=\sum_{i=1}^{r}g_{i}
\]
where the $g_{i}$ have partition rank $1$. The \emph{slice rank}
of $F$ is the minimal $r$ such that 
\[
F=\sum_{i=1}^{r}g_{i}
\]
where the $g_{i}$ have slice rank $1$.
\end{defn}
The partition rank is the minimal rank among all possible ranks obtained
from partitioning the variables. The slice rank can be viewed as the
rank which results from the partitions of $\{1,\dots,k\}$ into a
set of size $1$ and a set of size $k-1$, and so we have that 
\[
\prank\leq\slicerank.
\]
For two variables, the slice rank, partition rank, and tensor rank
are equivalent since there is only one non-trivial partition of a
set of size $2$. For three variables, the partition rank and the
slice rank are equivalent, and for $4$ or more variables, all three
ranks are different. A key property of the partition rank is the following
lemma, given in \cite[Lemma 11]{Naslund2020PartitionRank}, which
generalizes \cite[Lemma 1]{TaosBlogCapsets}.
\begin{lem}
\label{lem:Critical_Lemma}Let $X$ be a finite set, and let $X^{k}$
denote the $k$-fold Cartesian product of $X$ with itself. Suppose
that 
\[
F:X^{k}\rightarrow\mathbb{F}
\]
is the diagonal identity tensor, that is 
\[
F(x_{1},\dots,x_{k})=\delta(x_{1},\dots,x_{k})=\begin{cases}
1 & x_{1}=\cdots=x_{k}\\
0 & \text{otherwise}
\end{cases}.
\]
Then 
\[
\prank(F)=|X|.
\]
\end{lem}
\begin{proof}
See \cite[Lemma 11]{Naslund2020PartitionRank}.
\end{proof}

\subsection{The Distinctness Indicator}

Let $X$ be a finite set, $\mathbb{F}$ a field, and let $X^{k}=X\times\cdots\times X$
denote the Cartesian product of $X$ with itself $k$ times. $S_{k}$
acts on $X^{k}$ by permutation, that is for $\sigma\in S_{k}$, $(x_{1},\dots,x_{k})\in X^{k}$,
we have the group action $\sigma\cdot(x_{1},\dots,x_{k})=(x_{\sigma(1)},\dots,x_{\sigma(k)})$.
For every $\sigma\in S_{k}$, define 
\[
f_{\sigma}\colon X\times\cdots\times X\rightarrow\mathbb{F}
\]
to be the function that is $1$ if $(x_{1},\dots,x_{k})$ is a fixed
point of $\sigma$, and $0$ otherwise. We will make use of the following
Lemma, proven in \cite{Naslund2020PartitionRank}:
\begin{lem}
\label{lem:critical-indicator-function}Let $\text{Cyc}\subset S_{k}$,
be the $k$-cycles in $S_{k}$, and define 
\[
H_{k}(x_{1},\dots,x_{k})=\sum_{\begin{array}{c}
\sigma\in S_{k}\\
\sigma\notin\text{Cyc}
\end{array}}\sgn(\sigma)f_{\sigma}(x_{1},\dots,x_{k}).
\]
Then 
\[
H_{k}(x_{1},\dots,x_{k})=\begin{cases}
1 & \text{if }x_{1},\dots,x_{k}\text{ are distinct},\\
(-1)^{k-1}(k-1)! & \text{if }x_{1}=\cdots=x_{k},\\
0 & \text{otherwise}.
\end{cases}
\]
\end{lem}
\begin{proof}
See \cite[Lemma 15]{Naslund2020PartitionRank}.
\end{proof}
This function can be used to zero-out those tuples of vectors with
repetitions. Suppose that we are interested in the size of the largest
set $A\subset X$ that does not contain $k$ distinct vectors satisfying
some condition $\mathcal{K}$. Then if $F_{k}:X^{k}\rightarrow\mathbb{F}$
is some function satisfying
\[
F_{k}(x_{1},\dots,x_{k})=\begin{cases}
c_{1} & \text{if }x_{1},\dots,x_{k}\text{ satisfy }\mathcal{\mathcal{K}}\\
c_{2} & \text{if }x_{1}=\cdots=x_{k}\\
0 & \text{otherwise}
\end{cases}
\]
where $c_{2}\neq0$, then 
\[
I_{k}(x_{1},\dots,x_{k})\coloneqq F_{k}(x_{1},\dots,x_{k})H_{k}(x_{1},\dots,x_{k})
\]
when restricted to $A^{k}$ will be a diagonal tensor, and hence by
lemma \ref{lem:Critical_Lemma}
\[
|A|\leq\prank(I_{k}).
\]
Let 
\begin{equation}
\delta(x_{1},\dots,x_{k})=\begin{cases}
1 & \text{if }x_{1}=\cdots=x_{k}\\
0 & \text{otherwise}
\end{cases}\label{eq:delta_def}
\end{equation}
and for a partition $P$ of $\{1,\dots,k\}$ define 
\[
\delta_{P}(x_{1},\dots,x_{k})=\prod_{A\in P}\delta_{A}\left(\vec{x}_{A}\right).
\]
For every $\sigma$, we have that $f_{\sigma}=\delta_{P}$ for some
partition $P$, it follows that we can write
\begin{equation}
H_{k}(x_{1},\dots,x_{k})=\sum_{P\in\mathcal{P}_{k}}c_{P}\prod_{A\in P}\delta_{A}\left(\vec{x}_{A}\right)\label{eq:express_H_as_sum_product_delta_functions}
\end{equation}
where $\mathcal{P}_{k}$ is the set of non-trivial parititons of $\{1,\dots,k\}$
and $c_{P}$ are constants. Furthermore, since the definition of $H$
specifically excludes the cycles, it follows that only non-trivial
partitions will appear with non-zero coefficients in the sum above.
That is, every product of delta functions contains two or more terms
in the product.

\section{Lower Bounds for $\chi_{k}\left(\mathbb{R}^{n};m\right)$\label{sec:Lower-Bounds-for-chi_k}}

Let $S\subset\mathbb{R}^{n}$ be a finite set such that $\|x-y\|_{2}^{2}$
is even for any $x,y\in S$. Let $p$ be a prime satisfying
\[
p>\max_{x,y\in S}\frac{\|x-y\|_{2}^{2}}{2}.
\]
Consider the polynomials $F_{r}:S^{r}\rightarrow\mathbb{F}_{p}$ defined
by 
\[
F_{r}(x_{1},\dots,x_{r})=\prod_{i<j}\left(1-\left(\frac{\|x_{i}-x_{j}\|_{2}^{2}}{2}\right)^{p-1}\right),
\]
and note that 
\begin{equation}
\deg F_{r}=2\binom{r}{2}(p-1).\label{eq:degree_of_F_r}
\end{equation}
Our results are most naturally stated when the domain of $F_{r}$
is a cube-like subset of a lattice.
\begin{defn}
Let $v_{1},\dots,v_{n}$ be a basis of $\mathbb{R}^{n}$. We say that
a set $A\subset\mathbb{R}^{n}$ is an \emph{$l^{n}$-cube} if for
some constant $k\in\mathbb{R}^{n}$
\[
A=\left\{ k+c_{0}v_{0}+\cdots+c_{n}v_{n}:\ c_{i}\in\{0,\dots,l\}\right\} .
\]
\end{defn}
To prove the main result, we need only consider the specific $l^{n}$-cube
$\{0,1,\dots,l\}^{n}$, however we state our results in this section
with this generality to prove Theorem \ref{thm:theta_function_double_cap_conjecture_claim}
from Section \ref{sec:Open-Problems} in the Appendix. Note that if
$S$ is any subset of an $l^{n}$-cube, then 
\begin{equation}
\dim\left\{ f:S\rightarrow\mathbb{F},\ \deg f\leq d\right\} \leq\#\left\{ v\in\{0,\dots,l\}^{n}:\ \sum_{i=1}^{n}v_{i}\leq d\right\} .\label{eq:dimension_of_l_n_cube_space}
\end{equation}
Let $A_{m}=\left\{ 1,\sqrt{2,},\dots,\sqrt{m}\right\} $. Then 
\[
F_{k+1}(x_{1},\dots,x_{k+1})=\begin{cases}
1 & \|x_{i}-x_{j}\|_{2}\in\sqrt{2p}A_{m}\cup\{0\}\ \forall\ i,j\\
0 & \text{otherwise}
\end{cases}.
\]
In other words, $F_{k+1}$ is nonzero if and only if $x_{1},\dots,x_{k+1}$
form a (possibly degenerate) $k$-simplex with distances in $\sqrt{2p}A_{m}$.
Note that $p$ is a prime that depends on $m$. The distinctness indicator
function from the previous section will help construct an indicator
function that captures when $x_{1},\dots,x_{k+1}$ are a non-degenerate
simplex. Define the function 
\begin{equation}
J_{k}=H_{k+1}F_{k+1}.\label{eq:J_k definition_chromatic_cliques}
\end{equation}
Then $J_{k}$ will satisfy
\begin{equation}
J_{k}(x_{1},\dots,x_{k})=\begin{cases}
1 & \text{if }x_{1},\dots,x_{k+1}\text{ are distinct and form a }k\text{-simplex},\\
(-1)^{k}k! & \text{if }x_{1}=\cdots=x_{k+1},\\
0 & \text{otherwise}.
\end{cases}\label{eq:J_k_as_indicator_function_chromatic}
\end{equation}
The prime $p$ increases with the diameter of the set, so if $n$
is large enough, we have $p>k$ and hence $(-1)^{k}k!\not\equiv0\pmod{p}$.
In particular, if $A\subset S$ does not contain a $k$-simplex, and
$n$ is large enough to insure that $p>k$, we have that
\[
J_{k}|_{A^{k+1}}=(-1)^{k}k!\cdot\delta(x_{1},\dots,x_{k+1}).
\]
That is, $J_{k}$ restricted to $A^{k+1}$ is diagonal, and hence
by Lemma \ref{lem:Critical_Lemma}, $|A|\leq\prank(J_{k})$. 
\[
\chi_{k}(\mathbb{R}^{n},A_{m})\geq\frac{|S|}{\prank(J_{k})}.
\]

\begin{lem}
\label{lem:partition_rank_chromatic_J_k_bound}Let $S\subset\mathbb{R}^{n}$
be a subset of an $l^{n}$-cube such that $\|x-y\|_{2}^{2}\in2\mathbb{Z}$
for all $x,y\in S$, and let $p>\max_{x,y\in S}\frac{1}{2}\|x-y\|_{2}^{2}$.
Then the function $J_{k}:S^{k+1}\rightarrow\mathbb{F}_{p}$ defined
by $J_{k}=F_{k+1}H_{k+1}$ satisfies
\[
\prank(J_{k})\leq2^{k+1}\cdot\#\left\{ v\in\{0,\dots,l\}^{n}:\ \sum_{i=1}^{n}v_{i}\leq k(p-1)\right\} .
\]
\end{lem}
\begin{proof}
Let $\mathcal{P}_{k+1}$ denote the set of non-trivial partitions
of $\{1,\dots,k+1\}$. By lemma \ref{lem:critical-indicator-function},
and (\ref{eq:express_H_as_sum_product_delta_functions}), $H_{k+1}$
can be written as
\[
H_{k+1}=\sum_{P\in\mathcal{P}_{k+1}}c_{P}\prod_{A\in P}\delta\left(\vec{x}_{A}\right)
\]
where $c_{P}$ is a constant depending on the partition $P$. In the
construction of $H_{k+1}$, only non-trivial partitions were included,
and so we can split the product $J_{k}=H_{k+1}R_{k+1}$ into a linear
combination of terms of the form 
\begin{equation}
R_{k+1}(x_{1},\dots,x_{k+1})\prod_{A\in P}\delta\left(\vec{x}_{A}\right)\label{eq:RkA_term}
\end{equation}
where the product of delta functions term always contains two or more
delta functions. For such a term, the product of delta functions forces
many variables among $x_{1},\dots,x_{k+1}$ to be equal. For a set
$A$ and any element $a\in A$, and a polynomial function $Q$ of
$|A|$ variables, the product $\delta(\vec{x}_{A})Q(\vec{x}_{A})$
will be exactly equal to $\delta(\vec{x}_{A})Q(x_{a},x_{a},\dots,x_{a})=\delta(\vec{x}_{A})\tilde{Q}(x_{a})$,
where $\tilde{Q}$ is a single variable polynomial obtained from $Q$
by making all the variables equal. For our purposes, $a_{i}$ need
only be some representative member of $A_{i}$, where the choice is
made in a way that is consistent across different partitions. For
simplicity, we choose $a_{i}=\min(A_{i})$ for $i\leq r-1$ and $a_{r}=k+1$.
Let $P=\{A_{1},\dots,A_{r}\}$ be a non-trivial partition of $\{1,\dots,k+1\}$.
Let $a_{i}=\min(A_{i})$ denote the minimal element of $A_{i}$ for
each $i\leq r-1$, and suppose without loss of generality that $k+1\in A_{r}$,
and let $a_{r}=k+1$. Then 
\[
F_{k+1}(x_{1},\dots,x_{k+1})\prod_{A\in P}\delta(\vec{x}_{A})=F_{r}(x_{a_{1}},\dots,x_{a_{r}})\prod_{A\in P}\delta(\vec{x}_{A}).
\]
We can expand $F_{r}$ in the above as a sum of monomials of the form
\[
\left(\delta(\vec{x}_{A_{1}})x_{a_{1},1}^{\epsilon_{1,1}}\cdots x_{a_{1},n}^{\epsilon_{1,n}}\right)\cdots\left(\delta(\vec{x}_{A_{r}})x_{a_{r},1}^{\epsilon_{r,1}}\cdots x_{a_{r},n}^{\epsilon_{r,n}}\right)
\]
Since $F_{r}$ has degree $2\binom{r}{2}(p-1)$, for each such term,
we must have 
\[
\sum_{i=1}^{r}\sum_{j=1}^{n}\epsilon_{i,j}\leq\deg F_{r}\leq\binom{r}{2}(p-1),
\]
and so, there exists $1\leq i\leq r$ such that 
\[
\sum_{j=1}^{n}\epsilon_{i,j}\leq(r-1)(p-1).
\]
Proceeding in this manner for every non-trivial partition of $\{1,\dots,k+1\}$,
it follows that we can decompose 
\[
J_{k}=\sum_{\begin{array}{c}
A\subset\{1,\dots,k+1\}\\
1\leq|A|\leq k
\end{array}}\delta(\vec{x}_{A})\sum_{\begin{array}{c}
f\\
\deg f\leq k(p-1)
\end{array}}f(x_{a})G_{f}\left(\vec{x}_{\{1,\dots,k+1\}\backslash A}\right)
\]
where for each set $a=\min A$, and $G_{f}$ is some function, possibly
depending on $f$. This implies that 
\[
\prank(J_{k})\leq(2^{k+1}-2)\dim\left\{ f|\ f:S\rightarrow\mathbb{F}_{p}\text{ and }\deg f\leq k(p-1)\right\} ,
\]
and the result follows by (\ref{eq:dimension_of_l_n_cube_space})
since 
\[
\dim\left\{ f|\ f:S\rightarrow\mathbb{F}_{p}\text{ and }\deg f\leq k(p-1)\right\} \leq\#\left\{ v\in\{0,\dots,l\}^{n}:\ \sum_{i=1}^{n}v_{i}\leq k(p-1)\right\} .
\]
\end{proof}
\begin{rem}
In the statement of Lemma \ref{lem:partition_rank_chromatic_J_k_bound},
the constant in front depending on $k$ is crude. Proceeding more
carefully, taking into account each subset $A\subset\{1,\dots,k+1\}$,
in a similar manner to \cite[Lemma 22, Prop 23]{Naslund2020PartitionRank},
the bound can be improved to 
\[
\prank(J_{k})\leq(1+o_{k}(1))\#\left\{ v\in\{0,\dots,l\}^{n}:\ \sum_{i=1}^{n}v_{i}\leq k(p-1)\right\} ,
\]
where the $o_{k}(1)$ tends to zero as $n\rightarrow\infty$ with
a constant depending on $k$. However, for our purposes, Lemma \ref{lem:partition_rank_chromatic_J_k_bound}
as stated is sufficient.
\end{rem}
\begin{prop}
\label{prop:m_colors_chromatic_in_terms_of_dmax_cliques}Let $k\geq1$,
and let $m\geq1$ be the number of colors. Let $S$ be a subset of
an $l^{n}$-cube where $\|x-y\|_{2}^{2}\in2\mathbb{Z}$ for all $x,y\in S$.
Suppose that the maximum distance between elements of $S$ is 
\[
d_{\max}=\max_{x,y\in S}\frac{1}{2}\|x-y\|_{2}^{2}.
\]
Then 
\[
\chi(\mathbb{R}^{n},A_{m})\geq|S|2^{k+1}\max_{0<t<1}\frac{t^{\frac{k}{m+1}d_{\max}+\epsilon_{0}}}{\left(1+t+\cdots+t^{l}\right)^{n}}
\]
where $\epsilon_{0}\leq C(d_{\max})^{0.525}$ for a fixed constant
$C>0$, where $A_{m}=\{1,\sqrt{2},\dots,\sqrt{m}\}$.
\end{prop}
\begin{proof}
Let $p$ be the smallest prime satisfying $p>\frac{1}{m+1}d_{\max}$.
We will bound the maximum size of a set in $\mathbb{R}^{n}$ avoiding
any distances in $\{\sqrt{2p},\sqrt{4p},\dots,\sqrt{2mp}\}$, which
is a scaling of $A_{m}$. We have that
\[
p<\frac{1}{m+1}d_{\max}+\epsilon_{0}
\]
 where $\epsilon_{0}\leq C\left(\frac{d_{\max}}{m+1}\right)^{0.525}$
due to the best bounds on prime gaps \cite{BakerHarmanPintz2001DifferenceBetweenConsecutivePrimes}.
Let $J_{k}$ be defined as in (\ref{eq:J_k definition_chromatic_cliques}).
Then if $A\subset S$ does not contain a $(k+1)$-clique, we must
have that $J_{k}$ restricted to $A$ is diagonal, that is
\[
J_{k}|_{A^{k+1}}=\delta(x_{1},\dots,x_{k+1}).
\]
By Lemma \ref{lem:Critical_Lemma}, this implies that
\[
|A|\leq\prank(J_{k}),
\]
and by Lemma \ref{lem:partition_rank_chromatic_J_k_bound}, we have
\[
|A|\leq2^{k+1}\cdot\#\left\{ v\in\{0,\dots,l\}^{n}:\ \sum_{i=1}^{n}v_{i}\leq\frac{k}{m+1}d_{\max}+\epsilon_{0}\right\} .
\]
Consider the expansion of 
\[
\frac{(1+t+\cdots+t^{l})^{n}}{t^{C}}.
\]
The terms in this expansion are in one to one correspondence with
the vectors $v\in\{0,\dots,l\}^{n}$, and each vector $v$ with $\sum_{i=1}^{n}v_{i}\leq C$
will have a coefficient of $t$ with a negative power. It follows
that for $0<t<1$, we must have 
\[
\#\left\{ v\in\{0,\dots,l\}^{n}:\ \frac{1}{n}\sum_{i=1}^{n}v_{i}\leq\frac{k}{m+1}d_{\max}+\epsilon_{0}\right\} \leq\frac{(1+t+\cdots+t^{l})^{n}}{t^{\frac{k}{m+1}d_{\max}+\epsilon_{0}}},
\]
and hence 
\[
|A|\leq2^{k+1}\min_{0<t<1}\frac{\left(1+t+\cdots+t^{l}\right)^{n}}{t^{\frac{k}{m+1}d_{\max}+\epsilon_{0}}}.
\]
Let $\mathcal{A}$ be a coloring of $S$ without $(k+1)$-cliques.
That is, let $\mathcal{A}$ be collection of sets such that $S\subset\cup_{A\in\mathcal{A}}A$,
and such that each $A\in\mathcal{A}$ does not contain a $(k+1)$-clique
of elements with pairwise distances in $\sqrt{2p}A_{m}$. Then, due
to the upper bound on the possible size of each $A$, we must have
that 
\[
|\mathcal{A}|\geq\frac{|S|}{2}2^{k+1}\max_{0<t<1}\frac{t^{\frac{k}{m+1}d_{\max}+\epsilon_{0}}}{\left(1+t+\cdots+t^{l}\right)^{n}}.
\]
\end{proof}
\begin{lem}
\label{lem:multinomial_lower_bounding}Let $0<t<1$, and suppose that
$c=(c_{0},\dots,c_{l})\in\mathbb{R}_{\geq0}^{l}$ is fixed. Let 

\[
\mathcal{A}_{n.l}=\left\{ a\in\mathbb{N}^{l}:\ a_{0}+\cdots+a_{l}=n,\ a_{i}\geq0\right\} 
\]
 and 
\[
\mathcal{A}_{n.l}(c)=\left\{ a\in\mathcal{A}_{n,l}:\ a_{i}\leq a_{j}\text{ if }c_{i}\geq c_{j}\right\} .
\]
Then
\[
\max_{a\in\mathcal{A}_{n,l}(c)}\binom{n}{a_{0},\dots,a_{l}}t^{c_{0}a_{0}+\cdots+c_{1}a_{l}}\geq\frac{\left(t^{c_{0}}+\cdots+t^{c_{l}}\right)^{n}}{n^{l}}.
\]
\end{lem}
\begin{proof}
The multinomial theorem states that
\[
\sum_{a\in\mathcal{A}_{n,l}}\binom{n}{a_{0},\dots,a_{k}}t^{c_{0}a_{0}+\cdots+c_{l}a_{l}}=\left(t^{c_{0}}+\cdots+t^{c_{l}}\right)^{n},
\]
and so
\[
\max_{a\in\mathcal{A}_{n,l}}\binom{n}{a_{0},\dots,a_{l}}\geq\frac{1}{|\mathcal{A}_{n,l}|}\left(t^{c_{0}}+\cdots+t^{c_{l}}\right)^{n}.
\]
We trivially have the bound $|\mathcal{A}_{n,l}|\leq n^{l}$, since
$a_{l}$ is determined by $a_{0},\dots,a_{l-1}$, and since there
are at most $n$ choices for each of $a_{0},\dots,a_{l-1}$. What
remains to be shown is that the maximal $a$ lies in the restriction
of $\mathcal{A}_{n,l}$ to $\mathcal{A}_{n,l}(c)$. Indeed, since
$0<t<1$, we have that 
\[
t^{c_{i}}\leq t^{c_{j}}
\]
if $c_{i}\geq c_{j}$, and so, for the maximal tuple $a$, we must
have $a_{i}\leq a_{j}$, since if $a_{i}>a_{j}$, transposing $a_{i}$
and $a_{j}$ would increase the quantity being maximized.
\end{proof}
Using Lemma \ref{lem:multinomial_lower_bounding} and Proposition
\ref{prop:m_colors_chromatic_in_terms_of_dmax_cliques} we are now
ready to prove Theorem \emph{\ref{thm:main_theorem_specific_max_clique}. }
\begin{proof}
\emph{(of Theorem \ref{thm:main_theorem_specific_max_clique})} Let
$S_{n}^{a}\subset\{0,1,\dots,l\}^{n}$ be a set where each element
contains $a_{i}$ copies of element $i$, so that $\sum_{i=0}^{l}a_{i}=n$.
Then 
\[
|S_{n}^{a}|=\binom{n}{a_{0},\dots,a_{l}},
\]
and every element of $S_{n}^{a}$ will have norm of the same parity,
and hence any two elements $x,y\in S_{n}^{a}$ will have even distance
squared distance. Due to the convexity of $x^{2}$, the maximum squared
distance between two elements in $S_{n}^{a}$ occurs when the $l$'s
match up with the $0$'s, the remaining $0$'s match up with the $(k-1)$'s,
the remaining $(k-1)$'s match up with the $1$'s, and so on. For
ease of indexing, let $\{b_{i}\}_{i=0}^{l}$ be a permutation of the
$a_{i}$, where $b_{l}=a_{l}$, $b_{l-1}=a_{0}$, $b_{l-2}=a_{l-1}$,
$b_{l-3}=a_{1}$, and so on. Then it is possible to match coefficients
in precisely the way described above if for each $j$
\begin{equation}
\sum_{i>j}(-1)^{j+i+1}b_{i}\leq b_{j}.\label{eq:b_i_nonnegativity_requirement-1}
\end{equation}
Assume that $b_{i}\leq b_{j}$ for $i>j$, which implies this condition.
Let 
\[
d_{\max}=\max_{x,y\in S_{n}^{a}}\frac{1}{2}\|x-y\|_{2}^{2}.
\]
Then we have that 
\begin{align*}
d_{\max} & =l^{2}b_{l}+(l-1)^{2}(b_{l-1}-b_{l})+(l-2)^{2}(b_{l-2}-b_{l-1}+b_{l})+\cdots\\
 & =\sum_{i=0}^{l}(l-i)^{2}\sum_{j=l-i}^{l}(-1)^{j+l-i}b_{j}.
\end{align*}
Note that the assumption in (\ref{eq:b_i_nonnegativity_requirement-1})
is equivalent to assuming that the inner sum is always nonnegative.
Rearranging the order of summation, we have that

\begin{align*}
d_{\max} & =\sum_{j=0}^{l}b_{j}\sum_{i=l-j}^{l}(l-i)^{2}(-1)^{l-j-i}\\
 & =\sum_{j=0}^{l}b_{j}\sum_{i=0}^{j}i^{2}(-1)^{i}(-1)^{i+j}\\
 & =\sum_{j=0}^{l}b_{j}\binom{j+1}{2}
\end{align*}
where the final equality is due to the identity
\[
\sum_{i=0}^{j}(j-i)^{2}(-1)^{i}=\binom{j+1}{2}.
\]
Thus for $S_{n}^{a}\subset\{0,1,\dots,l\}^{n}$ satisfying $b_{i}\leq b_{j}$
for $i>j$, we have that 
\[
d_{\max}=\max_{x,y\in S}\frac{1}{2}\|x-y\|_{2}^{2}=\sum_{j=0}^{l}b_{j}\binom{j+1}{2},
\]
and hence by Proposition \ref{prop:m_colors_chromatic_in_terms_of_dmax_cliques}
\[
\chi_{k}(\mathbb{R}^{n},A_{m})\geq2^{k+1}|S_{n}^{a}|\max_{0<t<1}\frac{t^{\frac{k}{m+1}\sum_{j=0}^{l}c_{j}a_{j}+\epsilon_{0}}}{\left(1+t+\cdots+t^{l}\right)^{n}}
\]
where the constants $c_{j}$ are the binomials $\binom{j+1}{2}$.
Maximizing over all possible sets $S_{n}^{a}$, Lemma \ref{lem:multinomial_lower_bounding}
implies that 
\[
\max_{a\in\mathcal{A}_{n,l}}|S_{n}^{a}|t^{\frac{1}{m+1}\sum_{j=0}^{l}c_{j}a_{j}}\geq\frac{1}{n^{l}}\left(\sum_{j=0}^{l}t^{\frac{k\binom{j+1}{2}}{m+1}}\right)^{n},
\]
and hence 
\[
\chi_{k}\left(\mathbb{R}^{n},A_{m}\right)\geq2^{k+1}\left(\max_{0<t<1}\frac{\theta\left(t^{\frac{k}{m+1}},l+1\right)t^{\frac{\epsilon_{0}}{n}}}{1+t+\cdots+t^{l}}\right)^{n}.
\]
Since $\frac{\epsilon_{0}}{n}\ll_{l}n^{-0.475}$, $t^{\frac{\epsilon_{0}}{n}}=1+o(1)$,
and $2^{\frac{k+1}{n}}=1+O(\frac{1}{n})$ and so we have that
\[
\chi_{k}\left(\mathbb{R}^{n},A_{m}\right)\geq\left(\max_{0<t<1}\frac{\theta\left(t^{\frac{k}{m+1}},l+1\right)}{1+t+\cdots+t^{l}}+o(1)\right)^{n}
\]
which proves the theorem. 
\end{proof}

\section{Asymptotics\label{sec:Asymptotics}}

In this section, we prove Theorem \ref{thm:max_l_theta_gamma_lower_bound_gamma_chi}.
Let 
\[
\theta\left(t\right)=\sum_{n=1}^{\infty}t^{\frac{n(n-1)}{2}},
\]
and 
\[
\theta\left(t;l\right)=1+t+t^{3}+t^{6}+t^{10}+\cdots+t^{\binom{l}{2}},
\]
and define 
\[
F_{\gamma}(t,l)=\frac{\theta\left(t^{\gamma};l\right)}{1+t+\cdots+t^{l-1}}.
\]
We begin by proving that $\max_{0<t<1}F_{\gamma}(t,l)$ is non-trivial
for any $0<\gamma<1$. 
\begin{prop}
For $0<\gamma<1$ we have that 

\begin{equation}
\max_{0<t<1}F_{\gamma}(t,l)>1\label{eq:gamma_partial_maximization_greater_than_1}
\end{equation}
and
\begin{equation}
\max_{0<t<1}\left(1-t\right)\theta(t^{\gamma})>1.\label{eq:theta_gamma_minimization_greater_than_1}
\end{equation}
\end{prop}
\begin{proof}
Let $\eta=\frac{1}{\gamma}$, and define$G_{\eta}(t,l)=F_{\gamma}(t^{\eta},l)$.
Then 
\[
\max_{0<t<1}G_{\eta}(t,l)=\max_{0<t<1}F_{\gamma}(t,l).
\]
Observe that $G_{\eta}(0,l)=\frac{l}{l}=1$, and since $\eta>1$,
\[
\frac{d}{dt}G_{\eta}(t,l)\biggr|_{t=0}=1.
\]
This implies that there exists $\epsilon>0$ such $G_{\eta}(\epsilon,l)>G_{\eta}(0,l)=1$,
equation (\ref{eq:gamma_partial_maximization_greater_than_1}) follows.
To prove the second part, let 
\[
G_{\eta}(t)=(1-t^{\eta})\theta(t).
\]
Then $G_{\eta}(0)=1$, and 
\[
\frac{d}{dt}G_{\eta}(t)=\left(1-t^{\eta}\right)\left(\sum_{n=2}^{\infty}\binom{n}{2}t^{\binom{n}{2}-1}\right)-\eta t^{\eta-1}\theta(t).
\]
Since $\eta>1$, we have $\frac{d}{dt}G_{\eta}(0)=1$, and equation
(\ref{eq:theta_gamma_minimization_greater_than_1}) follows.
\end{proof}
Next we try to understand for which $l$ $\max_{0<t<1}F_{\gamma}(t,l)$
is largest, and then lower bound this quantity by a minimization in
terms of $\theta(t)$.
\begin{prop}
For $0<\gamma<1$, 
\[
\max_{0<t<1}F_{\gamma}(t,l)
\]
 is maximized by choosing $l<\frac{2}{\gamma}$.
\end{prop}
\begin{proof}
Suppose that $l\geq\frac{2}{\gamma}$, and let $0<t_{0}<1$ be chosen
so that $F_{\gamma}(t_{0},l)$ is maximized. The final coefficient
of $\theta\left(t_{0}^{\gamma};l\right)$ is $t_{0}^{\gamma\binom{l}{2}}$,
and since $\gamma\binom{l}{2}\geq l-1,$ it follows that $t_{0}^{\gamma\binom{l}{2}}\leq t_{0}^{l-1}$.
For $a,b,c,d>0$, if $\frac{c}{d}<\frac{a}{b}$, then $\frac{a+c}{b+d}<\frac{a}{b},$and
so since $F_{\gamma}(t_{0},l)>1$, by setting $a=\theta\left(t_{0}^{\gamma};l-1\right)$,
$b=1+\cdots+t_{0}^{l-2}$, $c=t_{0}^{\gamma\binom{l}{2}}$, $d=t_{0}^{l-1}$,
it follows that
\[
F_{\gamma}(t_{0},l-1)>F_{\gamma}(t_{0},l),
\]
which proves the proposition.
\end{proof}
\begin{prop}
\label{prop:theta_gamma_l_theta_gamma_bound}For $0<\gamma<1$, and
for $l\geq\frac{2}{\gamma}$, 
\end{prop}
\[
\max_{0<t<1}\frac{\theta\left(t^{\gamma};l\right)}{1+t+\cdots+t^{l-1}}\geq\max_{0<t<1}(1-t)\theta\left(t^{\gamma}\right).
\]

\begin{proof}
Let $0<t_{1}<1$ be such that $(1-t_{1})\theta(t_{1}^{\gamma})$ is
maximized. For any $j\geq l$, $\gamma\binom{j+1}{2}\geq j$, and
hence
\[
\sum_{j\geq l}t_{1}^{\gamma\binom{j+1}{2}}\leq\sum_{j\geq l}t_{1}^{j},
\]
since the inequality holds term by term. Since 
\[
(1-t_{1})\theta(t_{1}^{\gamma})=\frac{\theta(t_{1};l)+\sum_{j\geq l}t_{1}^{\gamma\binom{j+1}{2}}}{1+\cdots+t_{1}^{l-1}+\sum_{j\geq l}t_{1}^{j}},
\]
the result again follows from the fact that $\max_{0<t<1}F_{\gamma}(t,l)\geq1$if
$a,b,c,d>0$ and $\frac{c}{d}<\frac{a}{b}$, then$\frac{a+c}{b+d}<\frac{a}{b}.$ 
\end{proof}
To complete the proof of Theorem \ref{thm:max_l_theta_gamma_lower_bound_gamma_chi},
$0<\gamma<1$ we need a general lower bound for 
\[
\max_{0<t<1}(1-t)\theta(t^{\gamma}).
\]
 To do so, we first use the Poisson Summation Formula to establish
a functional equation for $\theta(t)$.
\begin{prop}
\label{prop:func_equ}For $x>0$, we have that 
\[
\theta\left(e^{-\pi x}\right)=\frac{e^{\frac{\pi x}{8}}}{\sqrt{2x}}\vartheta_{4}\left(e^{-\frac{2\pi}{x}}\right),
\]
where 
\[
\vartheta_{4}(q)=\left(1+2\sum_{n=1}^{\infty}(-1)^{n}q^{n^{2}}\right)
\]
is a Jacobi Theta function.
\end{prop}
\begin{proof}
We have that 
\[
\theta\left(e^{-2\pi x}\right)=\frac{1}{2}e^{\frac{\pi x}{4}}\vartheta_{2}\left(e^{-\pi x}\right)
\]
where $\vartheta_{2}$ is the Jacobi Theta function 
\[
\vartheta_{2}\left(e^{-\pi x}\right)=\sum_{n=-\infty}^{\infty}e^{-\pi x\left(n+\frac{1}{2}\right)^{2}}.
\]
Let $f(t)=e^{-\pi x\left(t+\frac{1}{2}\right)^{2}}$, and consider
the Fourier transform
\[
\mathcal{F}\left(f(t)\right)(\omega)=\int_{\mathbb{R}}e^{-2\pi it\omega}f(t)dt.
\]
Then 
\[
\mathcal{F}\left(f(t)\right)=\frac{1}{\sqrt{x}}\exp\left(-\frac{\pi\omega^{2}}{x}+i\pi\omega\right),
\]
and hence by the Poisson Summation formula
\[
\vartheta_{2}\left(e^{-\pi x}\right)=\frac{1}{\sqrt{x}}\vartheta_{4}\left(e^{-\frac{\pi}{x}}\right).
\]
\end{proof}
Using this functional equation, we lower bound the maximum.
\begin{thm}
\label{thm:max_1-t_theta_lower_bound}For $0<\gamma<1$ we have that
\begin{equation}
\max_{0<t<1}\left(1-t\right)\theta(t^{\gamma})>\Gamma_{\chi}\sqrt{\frac{1}{\gamma}}\label{eq:theta_gamma_root_constant_lower_bound}
\end{equation}
where
\[
\Gamma_{\chi}=\sqrt{\frac{\pi}{2}}\max_{0<u<\infty}\frac{1-e^{-u}}{\sqrt{u}}=0.7998308498\dots.
\]
\end{thm}
\begin{proof}
Let $\eta=\frac{1}{\gamma}$. By replacing $t$ with $t^{\eta}$,
we need to show that for $\eta>1$ 
\[
\max_{0<t<1}\left(1-t\right)\theta(t^{\eta})>\Gamma_{\chi}\sqrt{\eta}.
\]
For $\eta<1.56$ we have that $\Gamma_{\chi}\sqrt{\eta}<1$, and so
the result holds trivially by (\ref{eq:theta_gamma_minimization_greater_than_1}).
Substituting $t=e^{-\pi x}$, by proposition \ref{prop:func_equ},
we have 
\[
\max_{0<t<1}\left(1-t^{\eta}\right)\theta(t)=\max_{0<x<\infty}\left(1-e^{-\pi\eta x}\right)\frac{e^{\frac{\pi x}{8}}}{\sqrt{2x}}\vartheta_{4}\left(e^{-\frac{2\pi}{x}}\right).
\]
Letting $x=\frac{u}{\pi\eta}$, this becomes 
\[
\sqrt{\frac{\pi\eta}{2}}\max_{0<u<\infty}\frac{1-e^{-u}}{\sqrt{u}}e^{\frac{u}{8\eta}}\vartheta_{4}\left(e^{-\frac{2\pi^{2}\eta}{u}}\right).
\]
Since $\vartheta_{4}(q)$ is an alternating series, 
\[
\vartheta_{4}\left(e^{-\frac{2\pi^{2}\eta}{u}}\right)\geq1-2e^{-\frac{2\pi^{2}\eta}{u}}.
\]
For $0\leq u<25\eta$ we have that 
\[
e^{\frac{u}{8\eta}}\left(1-2e^{-\frac{2\pi^{2}\eta}{u}}\right)\geq1,
\]
and so 
\[
\max_{0<t<1}\left(1-t^{\eta}\right)\theta(t)\geq\sqrt{\frac{\pi\eta}{2}}\max_{0<u<25}\frac{1-e^{-u}}{\sqrt{u}}.
\]
The function $\frac{1-e^{-u}}{\sqrt{u}}$ is maximized at $u=1.25643\dots$,
which implies that 
\[
\max_{0<u<25}\frac{1-e^{-u}}{\sqrt{u}}=\max_{0<u<\infty}\frac{1-e^{-u}}{\sqrt{u}}=0.638172686\dots,
\]
and the result follows.
\end{proof}
Theorem \ref{thm:max_l_theta_gamma_lower_bound_gamma_chi} then follows
from Theorem \ref{thm:max_1-t_theta_lower_bound} and Proposition
\ref{prop:theta_gamma_l_theta_gamma_bound}. 

\section{Open Problems\label{sec:Open-Problems}}

In this section we discuss some open problems related to the chromatic
number of $\mathbb{R}^{n}$ with multiple forbidden distances.

\subsection*{Sphere Packing}

Kupavskii's upper bound from \cite{Kupavskii2010ChromaticNumberForbiddenDistances}
can be restated in terms of the lattice sphere packing constant. Let
$\Delta_{n}$ denote the density of the densest lattice packing of
spheres in $\mathbb{R}^{n}$. The best bounds are (see \cite{KabatianskyLevenshtein1978BoundsForPackingOnTheSphere})
\[
2^{-n(1+o(1))}\leq\Delta_{n}\leq2^{-0.5990n}.
\]
Kupavskii proved that
\[
\chi(\mathbb{R}^{n},[1,l])\leq\frac{(l+1+o(1))^{n}}{\Delta_{n}}.
\]
Theorem \ref{thm:main_theorem_specific_max_clique} along with Kupavskii's
bound for $l=\sqrt{m}$, implies that for $A_{m}=\{1,\sqrt{2},\dots,\sqrt{m}\}$,
\[
\Gamma_{\chi}\sqrt{m+1}\leq\frac{(\sqrt{m}+1+o(1))^{n}}{\Delta_{n}}.
\]
For $n$ sufficiently large, this implies the upper bound 
\[
\Delta_{n}\leq\frac{1+o(1)}{\Gamma_{\chi}^{n}}.
\]
Sadly this is a trivial result since $\Gamma_{\chi}<1$, however a
$26\%$ improvement to the constant in Theorem \ref{thm:main_theorem_chi_m}
would yield a non-trivial result.
\begin{problem}
Can the bound in Theorem \ref{thm:main_theorem_chi_m} be improved
to 
\[
\chi(\mathbb{R}^{n},A_{m})\geq(C\sqrt{m}+o(1))^{n}
\]
for some $C>1$, giving a new proof of a non-trivial upper bound for
sphere packing? 
\end{problem}

\subsection*{Upper Bounds}

Kupavskii's upper bound for $\overline{\chi}(\mathbb{R}^{n};m)$ is
far from the lower bound established in this paper. Can any improvement
be made? In particular, can the dependence on $m$ be reduced from
exponential to polynomial?
\begin{problem}
Does there exists $c_{1},c_{2}>0$ such that 
\[
\overline{\chi}\left(\mathbb{R}^{n};m\right)\leq\left(c_{1}m\right)^{c_{2}n}?
\]
\end{problem}
Kupavskii's upper bound for $\chi(\mathbb{R}^{n},A_{m})$ (\ref{eq:kupavskii_upper_bound})
differs from the Larman-Rogers bound when $m=1$. Indeed, when $m=1$,
(\ref{eq:kupavskii_upper_bound})\textbf{ }becomes
\[
\chi\left(\mathbb{R}^{n}\right)\leq(4+o(1))^{n}.
\]

\begin{problem}
Can Kupavskii's upper bound for $\chi(\mathbb{R}^{n},A_{m})$ be improved
in such a way that $m=1$ lines up with the Larman-Rogers bound of
$3^{n}$?
\end{problem}

\subsection*{Low Dimensions}

Let $g(n)$ denote the minimal number of distinct distances between
$n$ points in the plane. Then $\overline{\chi}(\mathbb{R}^{2};m)\geq g^{-1}(m)$.
Can this simple lower bound be improved at all for large $m$?
\begin{problem}
\label{prob:any_constant_greater_than}Does there exist $C>1$ such
that 
\[
\overline{\chi}(\mathbb{R}^{2};m)\geq Cg^{-1}(m)?
\]
\end{problem}
This is significantly weaker than Erd\H{o}s' question (see \cite[Problem 10 and 42]{Soifer2011BookChapterChromaticNumberAndItsHistoryChapter})
whether $\overline{\chi}(\mathbb{R}^{2};m)$ grows polynomially in
$m$. I believe that the answer to problem \ref{prob:any_constant_greater_than}
is ``yes'', but it seems like it would require a different approach
to this problem.

\subsection*{Theta Functions and the Double Cap Conjecture}

The \emph{double cap conjecture} states that the set $A\subset S^{n-1}$
with the largest volume that does not contain two orthogonal vectors
is the union of two spherical caps on either side of the sphere. Let
$V_{n-1}$ denote the maximum of $\text{vol}(A)/\text{vol}(S^{n-1})$
for any set $A\subset S^{n-1}$ avoiding two orthogonal vectors, and
define
\[
\mu_{S}=\limsup_{n\rightarrow\infty}\left(V_{n-1}\right)^{\frac{1}{n}}.
\]
Then we have that 
\[
\frac{1}{\sqrt{2}}\leq\mu_{S}\leq\frac{\sqrt{3}}{2},
\]
where the lower bound comes from the double cap, and the upper bound
is due to Raigorodskii \cite{Raigorodskii1999OnABoundInTheBorsukProblem}.
Let $\Lambda\subset\mathbb{R}^{d}$ be an even integral lattice, and
define 
\[
\theta_{\Lambda}(q)=\sum_{\lambda\in\Lambda}q^{-\|\lambda\|_{2}^{2}}.
\]
Define 
\[
\mu_{\Lambda}=\left(\max_{0<t<1}\theta_{\Lambda}(t)(1-t)^{d}\right)^{-\frac{1}{d}}.
\]
The methods of this paper yield the following result:
\begin{thm}
\label{thm:theta_function_double_cap_conjecture_claim}For any even
integral lattice $\Lambda$
\begin{equation}
\mu_{S}\leq\mu_{\Lambda}.\label{eq:theta_function_lambda_claim}
\end{equation}
\end{thm}
(See the Appendix for a proof) Let $D_{k}$ denote the elements of
$\mathbb{Z}^{k}$ at even distance from the origin, and let $\mu_{\mathbb{Z}}=\lim_{k\rightarrow\infty}\mu_{D_{k}}$.
Then by \cite[Chapter 4]{ConwaySloane1999LatticePackingBook} 
\[
\theta_{D_{n}}(t)=\frac{1}{2}\left(\theta_{3}(t)^{n}+\theta_{4}(t)^{n}\right)
\]
where $\theta_{3}(t)=\sum_{n=-\infty}^{\infty}t^{n^{2}}$ and $\theta_{4}(t)=\sum_{n=-\infty}^{\infty}(-1)^{n}t^{n^{2}}$,
and taking the limit as $n\rightarrow\infty$ we have 
\[
\mu_{\mathbb{Z}}=\min_{0<t<1}\left(\theta_{3}(t)(1-t)\right)^{-1}=0.883337\dots.
\]
Let $\Lambda_{24}$ denote the Leech Lattice, and $E_{8}$ the $E_{8}$
lattice. Using \cite[Chapter 4]{ConwaySloane1999LatticePackingBook}
we have the following calculations: 
\[
\mu_{E_{8}}=0.88406\dots
\]
and
\[
\mu_{\Lambda_{24}}=0.88407\dots.
\]
These all result in upper bounds for $\mu_{S}$ that are all worse
than Raigorodskii's bound of $\sqrt{3}/2=0.866\dots$.
\begin{problem}
Does there exist a lattice $\Lambda\subset\mathbb{R}^{d}$ such that
\[
\mu_{\Lambda}<\mu_{\mathbb{Z}}?
\]
\end{problem}
If yes, it maybe be possible to work with a subset of such a lattice
to obtain a new upper bound for the double cap conjecture.

\section*{Appendix: Lattices and The Double Cap Conjecture}

In Section \ref{sec:Open-Problems}, we looked at the \emph{double
cap conjecture}, and stated Theorem \ref{thm:theta_function_double_cap_conjecture_claim}
relating the problem to the theta function of a lattice. In this appendix,
we prove Theorem \ref{thm:theta_function_double_cap_conjecture_claim},
which is a similar proof to that of Theorem \emph{\ref{thm:main_theorem_specific_max_clique}}.
Let $\Lambda$ be an even integral lattice, and let $v_{1},\dots,v_{d}$
be a basis for $\Lambda$. For $l\in\mathbb{N}$ define 
\[
\Lambda_{l}=\left\{ c_{1}v_{1}+\cdots+c_{d}v_{d}:\ c_{i}\in\mathbb{Z},\ -l\leq c_{i}\leq l\right\} .
\]
Let $\ell=(|\Lambda_{l}|-1)/2$, and index the elements of $\Lambda_{l}$
from $-\ell$ to $\ell$, where elements indexed by $-i,i$ are antipodal
for every $i\neq0$. Then 
\[
\theta_{\Lambda}(t)=\sum_{\lambda\in\Lambda}t^{-\|\lambda\|_{2}^{2}}
\]
and given a choice of basis, we can define the truncation 
\[
\theta_{\Lambda}(t,l)=\sum_{i=-\ell}^{\ell}t^{-\|\lambda_{i}\|_{2}^{2}}.
\]

\begin{prop}
For any even integral lattice $\Lambda$ and any $l\geq1$, we have
that 
\[
\mu_{S}\leq\min_{0<t<1}\frac{1+t+\cdots+t^{2l+1}}{\theta_{\Lambda}(t,l)^{\frac{1}{d}}}
\]
for some choice of basis $v_{1},\dots,v_{d}$.
\end{prop}
\begin{proof}
Choose a basis for $\Lambda$. Consider the $n$-fold product of $\Lambda_{l}$
with itself, $\Lambda_{l}^{n}=\Lambda_{l}\times\cdots\times\Lambda_{l}\subset\mathbb{R}^{nd}$.
By the assumption that the lattice is integral and even, $\frac{\|x-y\|_{2}^{2}}{2}\in2\mathbb{Z}$
for every $x,y\in\Lambda_{l}^{n}$. Let $n$ be sufficiently large,
and let $S\subset\Lambda_{l}^{n}$ be such that each element contains
$a_{i}$ copies of element $i$ and $-i$ of $\Lambda_{l}$. Due to
the convexity of $x^{2}$, the maximum squared distance between two
elements in $S_{n}^{a}$ occurs when antipodal points are paired.
It follows that the maximum squared distance will be 
\[
\max_{x,y\in S}\frac{\|x-y\|_{2}^{2}}{2}=\sum_{i=-\ell}^{\ell}a_{i}\|v_{i}\|_{2}^{2},
\]
and this occurs exactly when $x,y$ are antipodal. Let $\tilde{S}$
be the set $S$ where directly antipodal points are removed, so that
$|\tilde{S}|>|S|/2$, and 
\[
d_{\max}=\max_{x,y\in\tilde{S}}\frac{\|x-y\|_{2}^{2}}{2}<\sum_{i=-\ell}^{\ell}a_{i}\|v_{i}\|_{2}^{2}.
\]
Let 
\begin{equation}
p=\frac{1}{2}\cdot\left(\sum_{i=-\ell}^{\ell}a_{i}\|v_{i}\|_{2}^{2}\right)\label{eq:definition_of_p}
\end{equation}
and suppose that $p$ is prime. Then the polynomial $F:\tilde{S}\times\tilde{S}\rightarrow\mathbb{F}_{p}$
defined by 
\[
F(x,y)=1-\left(\frac{\|x-y\|_{2}^{2}}{2}\right)^{p-1}
\]
 will be $1$ if $x=y$ or if $x,y$ are at a right angle. The proof
of Proposition \ref{prop:m_colors_chromatic_in_terms_of_dmax_cliques}
implies that if $A\subset S$ satisfies 
\[
|A|\geq4\min_{0<t<1}\frac{\left(1+t+\cdots+t^{2l+1}\right)^{nd}}{t^{d_{\max}}}
\]
then it contains two elements $x\neq y$ with $F(x,y)=1$, that is
two elements at a right angle. If we could maximize over all possible
sets $\tilde{S}$ while ignoring the constraint that $p$ must be
prime, then Lemma \ref{lem:multinomial_lower_bounding} would implies
that 
\[
\max_{\tilde{S}}|\tilde{S}|t^{n\sum_{i=-\ell}^{\ell}a_{i}\|v_{i}\|_{2}^{2}}\geq\frac{1}{nd}\left(\sum_{i=-\ell}^{\ell}t^{\|v_{i}\|_{2}^{2}}\right)^{n}=\frac{1}{nd}\theta_{\Lambda}(t,l)^{n},
\]
and hence that there exists $\tilde{S}$ such that if $A\subset\tilde{S}$
does not contain two elements at a right angle, then 
\[
\frac{|A|}{|\tilde{S}|}\leq4nd\min_{0<t<1}\frac{\left(1+t+\cdots+t^{2l+1}\right)^{nd}}{\theta_{\Lambda}(t,l)^{n}}.
\]
Using bounds for prime gaps, we can modify Lemma \ref{lem:multinomial_lower_bounding}
to handle the constraint that $p$ is prime, while only varying $n$
a small amount. This is due to the fact that small perturbations of
size $\ll n^{1-\epsilon}$ to the individual coefficients $a_{i}$
will not affect the final asymptotic, but do allow us to modify the
quantity 
\[
\frac{1}{2}\cdot\left(n\sum_{i=-\ell}^{\ell}a_{i}\|v_{i}\|_{2}^{2}\right)
\]
sufficiently to insure that it is prime. In particular, let $i$ be
such that $\alpha=\|v_{i}\|_{2}^{2}/2$ is minimal. If we increase
$n$ by $2$ and increase the coefficients $a_{i},a_{-i},$ corresponding
to $v_{i}$ and $-v_{i}$, each by $1$, we will increase the value
of $p$ by exactly $2\alpha$. Heilbronn \cite{Heilbronn1933PrimeGapsInArithmeticProgression}
showed that for any fixed modulus $q$, there exists $\delta>0$ such
that 
\[
\pi(x+x^{1-\delta};q,a)-\pi(x;q,a)\sim\frac{1}{\phi(q)}\frac{x^{1-\delta}}{\log x}
\]
where 
\[
\pi(x;q,a)=\sum_{\begin{array}{c}
p\leq x\\
p\equiv a\ (q)
\end{array}}1
\]
is the prime counting function, and $\phi(q)$ is the Euler-Totient
function (See also \cite[Page 141-142]{Narkiewicz2012RationalNumberTheoryForThe20thCentury}).
Applying this theorem with the modulus $2\alpha$, by modifying $a_{i},a_{-i}$
in the way described, it follows that there exists $\delta>0$ and
some constant $C_{1}$ depending on $v_{1},\dots,v_{d}$ and on $l$,
such that there exists $m$ such that 
\[
n\leq m\leq n+C_{1}n^{1-\delta}
\]
with $p=\frac{1}{2}\cdot\left(\sum_{i=-\ell}^{\ell}a_{i}\|v_{i}\|_{2}^{2}\right)$
prime. Consequently, there is a choice of $a_{i}$ such that for any
subset of $A\subset\tilde{S}$ of density 
\[
\frac{|A|}{|\tilde{S}|}\leq\left(\min_{0<t<1}\frac{\left(1+t+\cdots+t^{2l+1}\right)^{d}}{\theta_{\Lambda}(t,l)}+o(1)\right)^{n},
\]
and it follows that 
\[
\mu_{S}\leq\min_{0<t<1}\frac{1+t+\cdots+t^{2l+1}}{\theta_{\Lambda}(t,l)^{\frac{1}{d}}}.
\]
\end{proof}
For any $l$, since $1+t+\cdots+t^{2l+1}\leq\sum_{i=0}^{\infty}t^{i}$
for any $0<t<1$, it follows that 
\[
\min_{0<t<1}\frac{1+t+\cdots+t^{2l+1}}{\theta_{\Lambda}(t,l)^{\frac{1}{d}}}\leq\min_{0<t<1}\left(\theta_{\Lambda}(t,l)^{\frac{1}{d}}(1-t)\right)^{-1},
\]
and so 
\[
\mu_{S}\leq\min_{0<t<1}\left(\theta_{\Lambda}(t,l)^{\frac{1}{d}}(1-t)\right)^{-1}
\]
for every $l$. Taking the limit as $l\rightarrow\infty$, we obtain
Theorem \ref{thm:main_theorem_specific_max_clique}.

\specialsection*{Acknowledgements}

I would like thank Yufei Zhao for a number of insightful conversations,
and A.M. Raigorodskii for bringing certain references to my attention.
I would also like to thank the anonymous referee for a number of helpful
comments that improved this manuscript.

\bibliographystyle{amsplain}
\providecommand{\bysame}{\leavevmode\hbox to3em{\hrulefill}\thinspace}
\providecommand{\MR}{\relax\ifhmode\unskip\space\fi MR }
\providecommand{\MRhref}[2]{%
  \href{http://www.ams.org/mathscinet-getitem?mr=#1}{#2}
}
\providecommand{\href}[2]{#2}

\end{document}